\documentclass[a4paper, 12pt]{article}
\usepackage{amsmath, amsfonts,amsthm}
\usepackage{amssymb}
\usepackage{a4wide}

\usepackage{mathtools}
\usepackage{bbm}  

\usepackage{url}
\usepackage[hidelinks]{hyperref} 

\newtheorem{Theorem}{Theorem}
\newtheorem{Proposition}[Theorem]{Proposition}
\newtheorem{Lemma}[Theorem]{Lemma}

\newcommand\EE{{E}}

\newcommand\NN{\mathbb{N}}
\newcommand\PP{{P}}

\newcommand\RR{\mathbb{R}}

\newcommand\ZZ{\mathbb{Z}}
\newcommand\one{\mathbbm{1}} 

\newcommand{\tA}{{\widetilde A}}

\newcommand{\tE}{{\widetilde \EE}}

\newcommand{\tK}{{\widetilde K}}

\newcommand{\tT}{{\widetilde T}}
\newcommand{\tb}{{\tilde b}}

\newcommand{\tmu}{{\tilde \mu}}
\newcommand{\tpi}{{\widetilde \pi}}

\newcommand{\tr}{{\tilde r}}
\newcommand{\tsfX}{\widetilde\sfX}

\newcommand{\tik}{\check k}

\newcommand{\sfX}{\mathsf{X}}
\newcommand{\sfY}{\mathsf{Y}}
\newcommand{\sfZ}{\mathsf{Z}}

\newcommand{\sfx}{\mathsf{x}}

\newcommand{\sfz}{\mathsf{z}}

\newcommand{\hnu}{\hat{\nu}}

\newcommand{\tmax}{\textstyle{\max}}

\newcommand{\cX}{\mathcal X}

\newcommand\eps{\varepsilon}

\newcommand{\sprod}{{\textstyle\prod}}

\newcommand{\cdott}{{\,\cdot\,}}

\newcommand{\eqlaw}{\overset{\text{\tiny\rm law}}{=}}

\newcommand{\ton}{\,\underset{n\to\infty}{{\longrightarrow}}\,}
\newcommand{\tonw}{\,\underset{n\to\infty}{\overset{\text{\rm w}}{\longrightarrow}}\,}
\newcommand{\tol}{\,\underset{\lambda\to\lambda_*}{{\longrightarrow}}\,}

\allowdisplaybreaks

\begin{document}
\centerline{\large  Gibbs conditioning principle }
 \vspace{2mm}
\centerline{\large  for log-concave independent random variables}
\vspace{2mm}
\centerline{Eric Cator}
\vspace{1mm}
\centerline{Radboud University, Nijmegen}
\vspace{2mm}
\centerline{Pablo A. Ferrari}
\vspace{1mm}
\centerline{Universidad de Buenos Aires and Conicet}




\paragraph{Abstract}
Let $\nu_1,\nu_2,\dots$ be a sequence of probabilities on the nonnegative integers, and $\sfX=(X_1,X_2, \dots)$ be a sequence of independent random variables $X_i$ with law $\nu_i$. For $\lambda>0$ denote $Z^\lambda_i:= \sum_x \lambda^x\nu_i(x)$ and $\lambda^{\max}:= \sup\{\lambda>0:  Z^\lambda_i<\infty \text{ for all }i\}$, and assume $\lambda^{\max}>1$. 
For $\lambda<\lambda^{\max}$, define the tilted probability $\nu_i^{\lambda}(x):= \lambda^x\nu_i(x)/Z^{\lambda}_i$, and let $\sfX^\lambda$ be a sequence of independent variables  $X^\lambda_i$ with law $\nu^{\lambda}_i$, and denote $S^\lambda_n:= X^{\lambda}_1+\dots+X^{\lambda}_n$, with $S_n=S^1_n$. Choose $\lambda^*\in(1,\lambda^{\max})$  and denote $R^*_n:= \EE (S^{\lambda^*}_n)$. The Gibbs Conditioning Principle (GCP) holds if $P(\sfX\in\cdott|S_n>R^*_n)$ converges weakly to the law of $\sfX^{\lambda^*}$, as $n\to\infty$.  We prove the GCP for log-concave  $\nu_i$'s, meaning $\nu_i(x+1)\,\nu_i(x-1) \le ( \nu_i(x))^2$, subject to a technical condition that prevents condensation. The canonical measures are the distributions of the first $n$ variables, conditioned on their sum being $k$. Efron's theorem states that for log-concave $\nu_i$'s, the canonical measures are stochastically ordered with respect to $k$. This, in turn, leads to the ordering of the conditioned tilted measures $P(\sfX^\lambda\in\cdott|S^\lambda_n>R^*_n)$ in terms of $\lambda$. This ordering is a fundamental component of our proof.



\noindent
\paragraph{Published in} {\sl Markov Processes and Related Fields }   {\bf 32}, 63–77 (2026).
\href{10.61102/1024-2953-mprf.2026.32.1.00a}{doi:10.61102/1024-2953-mprf.2026.32.1.00a}

\section{Introduction}
\label{sec1}

Consider a sequence $X_i$ of independent not necessarily identically distributed random variables taking values on $\ZZ_+$, the nonnegative integers; denote $\nu_i$ the law of $X_i$ and assume that $\nu_i$ is log-concave for all $i$. Let $\lambda^{\max}:= \sup\{\lambda>0:\sup_i \EE (\lambda^{X_i})<\infty\}$. For $\lambda\in(0,\lambda^{\max})$ let $X^\lambda_i$ be a sequence of independent variables with marginals $P(X^\lambda_i=x)=\frac1{Z_i^\lambda}\lambda^x\nu_i(x)$, with $Z_i^\lambda:=\EE (\lambda^{X_i})$. Consider the sums $S_n=X_1+\dots+X_n$ and $S^\lambda_n=X^\lambda_1+\dots+X^\lambda_n$. Choose $\lambda^*\in(0,\lambda^{\max})$ and denote $R^*_n:= \EE (S^{\lambda^*}_n)$. To avoid condensation, we need a condition about the derivatives of the function $M_n(t) := \sum_{i=1}^n\log(Z^{e^t}_i)$ around $t^*=\log \lambda^*$. Our Theorem \ref{thm1} says that if $\lambda^*>1$ then, the variables $(X_i)_{i\ge1}$ conditioned to $S_n> R^*_n$ converge weakly (finite dimensional distributions) to $(X^{\lambda^*}_i)_{i\ge1}$, as $n$ grows to infinity. Analogous results hold for $\lambda^*<1$ and/or the condition $S_n=\lfloor R^*_n\rfloor$.

This problem is well studied for independent and identically distributed (iid) variables~$Y_i$. Let $\hnu_n$ be the empirical measure of the first $n$ variables, $\hnu_n\varphi:= \frac1n\bigl(\varphi(Y_1)+\dots+\varphi(Y_n)\big)$, for test functions $\varphi$. Denote by $\Pi$ a convex set of measures. In the case of iid random variables ($\nu_i\equiv \nu$), the Sanov property says that $\hnu_n$ satisfies $\lim_n \frac1n \log P(\hnu_n\in \Pi)= -D(\Pi|\nu)$, where $D(\Pi|\nu)=\inf_{\nu'\in \Pi} D(\nu'|\nu)$ and $D(\nu'|\nu):= \sum_x \log(\nu'(x)/\nu(x)) \nu'(x)$. Vasicek \cite{zbMATH03664088} showed that the distribution of $Y_1$ given $\hnu_n\in\Pi$ converges to $\nu^*$ as $n\to\infty$, where $\nu^*$ satisfies $D(\Pi|\nu)=D(\nu^*|\nu)$, for some specific choices of $\Pi$. Czisar \cite{MR744233} shows that the variables $(Y_1,\dots,Y_\ell)$ converge in distribution to independent random variables with marginal distribution $Y_i\sim\nu^*$ for a wider set of models. This property was later called Gibbs conditioning principle by Dembo and Zeitouni \cite{zbMATH00410740}, see also Sepp{\"a}l{\"a}inen and Rassoul-Agha \cite{zbMATH06427428}. 
 L\'eonard and Najim  \cite{zbMATH01877108} proved the GCP using the Sanov theorem. 
 Diaconis and Freedman \cite{zbMATH04069931} consider the exponential family, and Liggett \cite{MR238373} the case of stable laws.

The Markov case was discussed by  Bolthausen and Schmock \cite{MR1027105},  Dobrushin and Hryniv \cite{MR1402652} and Meda and Ney \cite{zbMATH01405915}, see also Dembo and Zeitouni \cite{zbMATH00849085}. Messer and Spohn \cite{MR704588} and Deuschel, Stroock and Zessin  \cite{zbMATH04201256} prove similar results for classical statistical mechanic models with pair interactions. 

When there is no exponential moment the extra mass may go to a single site, a phenomenon known as condensation. See for instance Evans and Waclaw \cite{evans-waclaw}, Armend\'ariz and Loulakis~\cite{zbMATH05908027}. Godr\`eche~\cite{zbMATH07232704} and  Ferrari, Landim and Sisko  \cite{zbMATH05204121}. 

The proof of Theorem \ref{thm1} has two ingredients. The first is Proposition \ref{lem1} which establishes the stochastic order in $\lambda$ of the conditioned tilted measures $P(\sfX^\lambda_n\in\cdott|S^\lambda_n> R)$, for fixed $n$ and $R$. In turn this is a consequence of Efron theorem \cite{MR171335}, which says that under log-concavity the canonical measures $\mu_n(\cdott|k):=P(\sfX_n\in\cdott|S_n=k)$ are stochastically ordered in~$k$, for each~$n$. The second ingredient is that the conditioned measures $P(\sfX^\lambda_\ell\in\cdott|S^\lambda_n> R^*_n)$ converge as $n\to\infty$  to the unconditioned measure $P(\sfX^\lambda_\ell\in\cdott)$ when the conditioning set gets full probability, $P(S^\lambda_n> R^*_n)\to_n1$, for $\lambda>\lambda^*$. This limit holds under condition \eqref{eq:cond}, as we show in Lemma \ref{lem2}. Finally, the theorem follows from a sandwich argument, the conditioned measure is bounded above and below by unconditioned tilted measures, which converge to $P(\sfX^{\lambda^*}_\ell\in\cdott)$, as $\lambda\to\lambda^*$.

We do not make use of the entropy approach. 

The paper is organized as follows. Section \ref{sec2} introduces the notation and states the Gibbs conditioning principle for our setup. In Section \ref{sec3} we present the order in $\lambda$ of the conditioned tilted measures and the convergence of the conditioned tilted measures to the unconditioned ones; we then show that the GCP is a consequence of these results. In Section~\ref{sec4} we show that the coupling proof of Efron's theorem given by Liggett \cite{zbMATH01905957} can be extended to a larger set of measures, and establish the order of the conditioned tilted measures.

\section{Gibbs conditioning principle}\label{sec2}

In this section we introduce notation and state our main result, Theorem \ref{thm1}.

Let $\nu_1,\nu_2,\dots$ be probability measures on $\ZZ_+:=\{0,1,2,\dots\}$.
Denote
\begin{align*}
  Z^\lambda_i&:= \sum_z \lambda^z\nu_i(z),\qquad
   \lambda^{\max}
  :=\sup\{\lambda>0: Z^\lambda_i <\infty, \text{ for all }i\}. 
\end{align*}
For $\lambda\in(0,\lambda^{\max})$ define the $\lambda$-tilted measures $\nu^\lambda_i$ on $\ZZ_+$ by 
\begin{align*}
  \nu^\lambda_i(x)
  &:= \frac{\lambda^x \,\nu_i(x)}{Z^\lambda_i}.
\end{align*}     
Denote by $\sfX^\lambda=(X_1^\lambda,X_2^\lambda\dots)$, a family of independent random variables with marginals $X_i^\lambda \eqlaw \nu_i^\lambda$; denote $X_i =X^\lambda_i$ and $\sfX = (X_1,X_2,\dots)$. For each natural $n$ denote 
\begin{align*}
  S^\lambda_n := X^\lambda_1+\dots+X^\lambda_n,\quad S_n:= S^1_n.
\end{align*}
Fix $\lambda^* \neq 1$ with $0<\lambda^*<\lambda^{\max}$ and define for each $n\geq 1$
\begin{align*}
  R^*_n := \EE(S^{\lambda^*}_n).
\end{align*}
We are interested in the limiting distribution of $\sfX=(X_1,X_2,\dots)$ when we condition on $\{S_n > R^*_n\}$ if $\lambda^*>1$, or on $\{S_n < R^*_n\}$ if $\lambda^*<1$. 


Before we can give our main result, we need a technical condition which is closely related to ensuring that the variance of $S_n^{\lambda^*}$ grows to infinity: define the functions
\begin{align*}
  M_n(t) = \sum_{i=1}^n\log(Z^{e^t}_i) = \sum_{i=1}^n \log\bigl(\EE(e^{tX_i})\bigr).
\end{align*}
Also define $t^*=\log(\lambda^*)$. It is well known that $M_n$ is a convex function, and it is not hard to see that since $\lambda^*<\lambda^{\max}$, we have that $M_n$ is twice differentiable at $t^*$, but we need a stronger condition around $t^*$: 
\begin{align}
    &\forall \eps\in(0,\log(\lambda^{\max})-t^*):\nonumber \\
  \begin{split}
    &\; \hspace{-10pt} M_n(t^*-\eps) - M_n(t^*) + \eps M'_n(t^*) \to +\infty,\mbox{ and}\;\\
   &\; \hspace{-10pt} M_n(t^*+\eps) - M_n(t^*) - \eps M'_n(t^*) \to +\infty.
 \end{split}\label{eq:cond}
 \end{align}
 Note that the convexity of $M_n$ implies that both expressions are increasing in $\eps$. Since $M''(t^*)$ is equal to the variance of $S_n^{\lambda^*}$, the condition is often implied by this variance going to $+\infty$, but this might not be enough: maybe the convex functions $M_n$ approach another convex function with a discontinuity in the derivative at $t^*$. This might not satisfy condition \eqref{eq:cond}.

 Denote $\max_i=\sup\{x\in\ZZ_+: \nu_i(x)>0\}$, which could be infinite. 
We say that $\nu_i$  is log-concave if $\nu_i(x)>0$ for all $x\in [0,\max_i]\cap\ZZ$, and 
\begin{align*}
    \frac{\nu_i(x-1)}{\nu_i(x)}\,\le\, \frac{\nu_i(x)}{\nu_i(x+1)},\quad 0< x < \tmax_i.
\end{align*}

\begin{Theorem}[Gibbs conditioning principle]
  \label{thm1} Let $\nu_1,\nu_2,\dots$ be a family of log-concave probability measures on~$\ZZ_+$. Let $\sfX=(X_1,X_2,\dots)$ be independent random variables with  $X_i\eqlaw \nu_i$. Choose $\lambda^*\in(0,\lambda^{\max})$, let $R^*_n=\EE(S^{\lambda^*}_n)$ and assume \eqref{eq:cond}. Then, the following limits hold.
\begin{align}
  &\text{If }\quad 1< \lambda^*<\lambda^{\max},\text{ then }
  &P(\sfX\in\cdott \,|\, S_n>R^*_n)
                       &\;\tonw\; P(\sfX^{\lambda^*}\in\cdott), \label{wek1}\\
  &\text{If } \quad 0<\lambda^*<1, \text{ then }
  & P(\sfX\in\cdott \,|\, S_n<R^*_n)
                       &\;\tonw\; P(\sfX^{\lambda^*}\in\cdott),\label{wek2}\\
  &\text{If }\quad 0<\lambda^*<\lambda^{\max},\text{ then }
  & P(\sfX\in\cdott \,|\, S_n=\lfloor R^*_n\rfloor)
                       &\;\tonw\; P(\sfX^{\lambda^*}\in\cdott), \label{wek3}
\end{align}
where $\tonw$ means convergence of the finite dimensional distributions, or weak convergence. 
\end{Theorem}
The theorem is proven at the end of Section \ref{sec2}.  It is helpful to think of a continuous example where each $\nu_i$ corresponds to a normal random variable with expectation $a_i$ and variance $\sigma_i^2$. This means that 
\begin{align*}
  R^*_n= \sum_{i=1}^n a_i + \lambda^*\sum_{i=1}^n\sigma_i^2\quad \text{and}\quad M_n(t) =  \sum_{i=1}^n a_it +  \frac12\sum_{i=1}^nt^2\sigma_i^2
\end{align*}

In this case Theorem 1 is true if $\sum_{i\geq 1} \sigma_i^2 = +\infty$ (since in that case, conditioning on the sum, which can be done explicitly, does not change the variance of the first $k$ normal random variables, as $n\to \infty$). However, if $\sum_{i\geq 1} \sigma_i^2 < +\infty$, then the marginal of the first random variable does not converge to the tilted random variable $X^{\lambda^*}_1$, since in that case conditioning on the sum changes the variance of all variables, whereas tilting a normal random variable leaves the variance constant. Also Condition \eqref{eq:cond} is not true if the total variance is finite (and it is true if the total variance is infinite).

Log-concave distributions are connected with several areas of mathematics and statistics, see the survey by Saumard and Wellner \cite{zbMATH06386273}.



\section{Conditioned tilted measures}
\label{sec3}

\paragraph{Stochastic order}
Consider the coordinatewise order in $\ZZ^n$,   $\sfx_n\le\sfx_n'$ if $x_i\le x'_i$,  $\forall i$.
A measure $\mu'$ stochastically dominates $\mu$, denoted 
$\mu\prec\mu'$, 
if there exists $\tmu$ on $\ZZ^n\times\ZZ^n$ with marginals $\mu$ and $\mu'$ such that 
a vector $(\sfX,\sfX')$ with distribution $\tmu$ satisfies
$\sfX\le\sfX'$ $\tmu$-a.s.. 

We next show that the distribution of independent random variables with $\lambda$ tilted log-concave marginals, conditioned to the event ``the sum of the first $n$ variables is bigger than $R$'', are stochastically ordered on $\lambda$ and $R$.
\begin{Proposition}[Stochastic order of conditioned measures]
  \label{lem1}
Let $\nu_1,\!\nu_2,\! ...$ be log-concave probability measures on~$\ZZ_+$ and $X_1,X_2,\dots$ be independent random variables with  $X_i\eqlaw \nu_i$. Let $n\in \NN$, and $R$ be a real number satisfying  $0<R<R^{\max}_n:=\sup_{\lambda<\lambda^{\max}} \EE (S^\lambda_n)$. Then, for 
$0<\lambda <\lambda'<\lambda^{\max}$ we have
  \begin{align}
    P( \sfX^\lambda\in\cdott|S^\lambda_n>R)&\;
                                                 \prec\;
                                                 P(\sfX^{\lambda'}\in\cdott|S^{\lambda'}_n>R),\label{ord1} \\
    P( \sfX^\lambda\in\cdott|S^\lambda_n<R)&\;
                                                 \prec\;
                                                 P(\sfX^{\lambda'}\in \cdott|S^{\lambda'}_n<R),\label{ord2}
    \\
    P(\sfX^\lambda\in\cdott|S^\lambda_n<R)&\;
                                                \prec\;
                                            P(\sfX^{\lambda'}\in\cdott|S^{\lambda'}_n>R).\label{ord3}
  \end{align}
Furthermore, for $0<\lambda<\lambda^{\max}$ and $k\in \NN$ satisfying $P(S_n=k)>0$, we have
  \begin{align}
   P(\sfX^{\lambda}\in\cdott|S^{\lambda}_n\le k)  \prec\; P(\sfX^\lambda\in\cdott|S^\lambda_n=k)\;
                                                \prec\;
                                                P(\sfX^{\lambda}\in\cdott|S^{\lambda}_n\ge k).\label{2}  
  \end{align}
\end{Proposition}

The proposition is proven in Section \ref{sec31} using two ingredients. The first is Efron's Theorem \cite{MR171335},  which states that the canonical measures $\mu_n(\cdott|k)$ of product measures with log-concave marginals are stochastically ordered in $k$; see Liggett \cite{zbMATH01905957} for a proof using monotonicity of the zero range process. The second says that the conditioned sum distributions $P(S^\lambda_n\in \cdott|S^\lambda_n\in I)$ are stochastically ordered in $\lambda$, for any interval $I$, a one-dimensional statement. 

\paragraph{Convergence of conditioned tilted measures}
We now show that under Condition \eqref{eq:cond} 
the conditioned tilted measure converges to the unconditioned one as the conditioning set gets full probability. The following lemma does not use log-concavity.
\begin{Lemma}
  \label{lem2}
Assume Condition \eqref{eq:cond}.  Then 
  \begin{alignat}{2}
    \text{If }\lambda&\in(\lambda^*,\lambda^{\max}),& \text{  then }
    P(\sfX^\lambda\in\cdott|S^\lambda_n>R^*_n) &\;\tonw\;  P(\sfX^\lambda\in\cdott),\label{ll1}\\
    \text{If }\lambda&\in(0,\lambda^*),& \text{  then }
    P(\sfX^\lambda\in \cdott|S^\lambda_n<R^*_n) &\;\tonw\;   P(\sfX^\lambda\in\cdott).\label{ll2}
  \end{alignat}
\end{Lemma}
\begin{proof}
Suppose $\lambda>\lambda^*$. We want to prove that $\PP(S_n^\lambda > R^*_n)\to 1$. Define $t=\log(\lambda)$ and $t^*=\log(\lambda^*)$. Define $\eps=t-t^*$.
\begin{align*}
\PP(S_n^\lambda \leq R^*_n) &= \PP(e^{-\eps S^\lambda_n}\geq e^{-\eps R^*_n}) \leq \frac{\EE(e^{-\eps S^\lambda_n})}{e^{-\eps R^*_n}} = \frac{\EE(e^{(t-\eps)S_n})}{e^{-\eps R^*_n}\,\EE(e^{t S_n})} .
\end{align*}
Therefore,
\begin{align*}
\log(\PP(S_n^\lambda \leq R^*_n)) & \leq M_n(t^*) - M_n(t^*+\eps) + \eps R^*_n\\
& = M_n(t^*) - M_n(t^*+\eps) + \eps M'_n(t^*) \to -\infty.
\end{align*}
For the last step we use Condition \eqref{eq:cond} and the fact that $M_n'(t^*) = \EE(S_n^{\lambda^*})=R^*_n$. This proves \eqref{ll1}. The proof for \eqref{ll2} is completely similar.
\end{proof}

\begin{proof}[of Theorem \ref{thm1}] We prove that the conditioned measure in \eqref{wek1} is dominated from above and below by measures converging to the distribution of $\sfX^{\lambda^*}$. Since we want the convergence of the finite dimensional distributions, it suffices to work with the marginal distribution of the vector $\sfX_\ell=(X_1,\dots,X_\ell)$ for each $\ell\ge1$. The upperbound for $\lambda^*<\lambda<\lambda^{\max}$ and $n\ge \ell$ is obtained by
\begin{align}
    \label{sand1}
   P(\sfX_\ell\in\cdott \,|\, S_n>R^*_n)\,
  &\prec\, P(\sfX^\lambda_\ell\in\cdott \,|\, S^\lambda_n>R^*_n)  \nonumber \\
  &\ton\, P(\sfX^\lambda_\ell\in\cdott)\, \tol\, P(\sfX^{\lambda^*}_\ell\in\cdott), 
\end{align}
where the domination comes from \eqref{ord1} as $\sfX_n=\sfX^1_n$, the $\lim_n$ from \eqref{ll1}, and the $\lim_\lambda$ from the continuity of $P(\sfX^\lambda_\ell\in\cdott)$ in $\lambda$. 
For the lowerbound, take $1<\lambda<\lambda^*$ to get
\begin{align}
    P(\sfX_\ell\in\cdott \,|\, S_n>R^*_n)\,
  &\succ\, P(\sfX^\lambda_\ell\in\cdott \,|\, S^\lambda_n<R^*_n) \nonumber \\
    &\ton\, P(\sfX^\lambda_\ell\in\cdott)\,
    \tol\, P(\sfX^{\lambda^*}_\ell\in\cdott), \label{21}
\end{align}
where we used \eqref{ord3} for the domination, \eqref{ll2} for the $\lim_n$, and continuity for the $\lim_\lambda$.  The bounds \eqref{sand1} and \eqref{21} show \eqref{wek1}.

To show  \eqref{wek2} proceed in the same way. For the upperbound let  $\lambda^* <\lambda<1$,
\begin{align} 
    P(\sfX_\ell\in\cdott \,|\, S_n<R^*_n)\,
  &\prec\, P(\sfX^\lambda_\ell\in\cdott \,|\, S_n>R^*_n) \nonumber \\
    &\ton\, P(\sfX^\lambda_\ell\in\cdott)\,
    \tol\, P(\sfX^{\lambda^*}_\ell\in\cdott). \label{sand2}
\end{align}
For the lowerbound, let $0<\lambda<\lambda^*$, 
\begin{align}
    P(\sfX_\ell\in\cdott \,|\, S_n<R^*_n)\,
  &\succ\, P(\sfX^\lambda_\ell\in\cdott \,|\, S^\lambda_n<R^*_n) \nonumber \\
    &\ton\, P(\sfX^\lambda_\ell\in\cdott)\,
    \tol\, P(\sfX^{\lambda^*}_\ell\in\cdott). \label{sand4}
\end{align}
The bounds \eqref{sand2} and \eqref{sand4} show \eqref{wek2}.

To show \eqref{wek3} use \eqref{2} to get
\begin{align}
  \label{can4}
  P(\sfX_\ell\in\cdott  \,|\, S_n\le \lfloor R^*_n \rfloor )
  &\prec
  P(\sfX_\ell\in\cdott  \,|\, S_n= \lfloor R^*_n \rfloor) \nonumber \\
  &\prec
  P(\sfX_\ell\in\cdott \,|\, S_n> \lfloor R^*_n \rfloor ),
\end{align}
so that \eqref{sand2} and \eqref{sand2} imply \eqref{wek3}.
\end{proof}

\section{Measure dominations}
\label{sec4}

In this section we show Proposition \ref{lem1}.
\subsection{Mixture of canonical measures}
\label{sec31}

Denote by $\pi^\lambda_n$ the distribution of the sum $S^\lambda_n$ and $\pi^\lambda_n(k|I)$ the distribution of the sum conditioned to belong to an integer set $I$, 
\begin{align*}
  \pi^\lambda_n(k):= \PP(S^\lambda_n=k),\qquad  \pi^\lambda_n(k|I)&:= \frac{\pi^\lambda_n(k)}{\pi^\lambda_n(I)}.
\end{align*}
Define the canonical measure $\mu_n(\cdott|k)$ as the distribution of $\sfX_n$ conditioned to $S_n=k$,
\begin{align}
  \label{can}
  \mu_n(\cdott|k) &:=  \frac1{\pi_n(k)}\,\PP(\sfX_n=\cdott,S_n=k).
\end{align}
The distribution of $\sfX^\lambda_n$ conditioned to $S^\lambda_n=k$ coincides with the canonical measure for each $\lambda<\lambda^{\max}$,
\begin{align*}
  \frac1{\pi^\lambda_n(k)}\,
  \PP(\sfX^\lambda_n   =\cdott,S^\lambda_n=k)
  = \frac{\sfZ^\lambda_n}{\lambda^k\pi_n(k)}
  \frac{\lambda^k}{\sfZ^\lambda_n}\,
  \PP(\sfX_n =\cdott,S_n=k)= \mu_n(\cdott|k),
\end{align*}
where $\sfZ^\lambda_n:=\sum_\ell \lambda^\ell P(S_n=\ell)$. 
This implies that the conditioned tilted measure is a mixture of canonical measures,
\begin{align}
  \PP(\sfX^\lambda_n=\cdott\,|\,S^\lambda_n\in I)
  &= \sum_{k \in I} \pi^\lambda_n(k|I)\; \mu_n(\cdott|k),\quad I\subset\ZZ_+.\label{can4}
\end{align}
In statistical mechanics this is called the grand canonical measure.
The proof of Proposition~\ref{lem1} is based on the next two results.
\begin{Lemma}[Order of one-dimensional tilted measures]\label{lem15}
Assume $n>0$ and $0<\lambda\le\lambda'<\lambda^{\max}$. 
Let $I$ be an integer interval $I\subset[0,R_n^{\max}]\cap\ZZ$. Then,
  \begin{align}
    \pi^\lambda_n(\cdott|I)&\;\prec\; \pi^{\lambda'}_n(\cdott|I) \label{pi6}
  \end{align}
\end{Lemma}
Lemma \ref{lem15}, proven in Section \ref{sec41}, does not require log concavity.

\begin{Theorem}[Efron \cite{MR171335}. Order of canonical measures]
  \label{lem16}
Let $X_i$ be a sequence of independent random variables  with log-concave marginals $X_i\eqlaw\nu_i$. Then, the canonical measures are stochastically ordered: for $n>0$ and $k\ge0$ such that $\pi_n(k+1)>0$, we have  
  \begin{align}
\mu_n(\cdott|k)\, &\prec\, \mu_n(\cdott|k+1).\label{3}
  \end{align}
\end{Theorem}

We provide a proof and more references in Section \ref{sec5}.

\begin{proof}[of Proposition \ref{lem1}]
Take $\lambda\in(\lambda^*,\lambda^{\max})$ and denote $I_n:= (R^*_n,R_n^{\max}]\cap \ZZ$. Use  Lemma~\ref{lem15} to construct a coupling $(K,K')$ satisfying that 
$K$ and $K'$ have marginal distributions $\pi^\lambda_n(\cdott|I_n)$, $\pi^{\lambda'}_n(\cdott|I_n)$, respectively, and
$K\le K'$.
Use Theorem \ref{lem16} to construct a coupling $(\sfY^k_n)_{k\ge0}$ such that the marginal $\sfY^k_n$ has canonical distribution $\mu_n(\cdott|k)$,  and $\sfY^k_n\le \sfY^{k+1}_n$ for all $k\ge1$.

Consider the vector $(K,K')$ independent of the family $(\sfY^k_n)_{k\ge0}$ and notice that the  coupling $(\sfY^K_n,\sfY^{K'}_n)$ satisfies
\begin{align*}
  \sfY^{K}_n &\eqlaw P(\sfX^\lambda_n\in\cdott| S^\lambda_n>R^*_n),\\
  \sfY^{K'}_n &\eqlaw P(\sfX^{\lambda'}_n\in\cdott| S^{\lambda'}_n>R^*_n), \\
  \sfY^K_n&\le \sfY^{K'}_n.
\end{align*}
This proves \eqref{ord1} because $\sfX^\lambda_n$ is independent of $(X^\lambda_i:i>n)$.

The same proof works for \eqref{ord2}. To see \eqref{ord3} observe that both measures are mixtures of canonical distributions, and the measure conditioned to $S^\lambda_n<R^*_n$ concentrates on canonical measures with $k<R^*_n$, while the measure on the right concentrates on canonicals with $k>R^*_n$, so that the order of the canonicals implies the order of those measures.

To see \eqref{2} observe that its middle term is the $k$ canonical measure $\mu_n(\cdott|k)$ while the right term is a mixture of $k'$ canonical measures with
$k'\ge k$,
$$\sum_{k'\ge k} \pi^\lambda_n(k'|[k,\infty)) \mu(\cdott|k'),$$
by \eqref{can4}. Hence, the second inequality in \eqref{2} is an immediate consequence of \eqref{3}. The same argument shows the first inequality.
\end{proof}

\subsection{Order of one dimensional tilted measures}
\label{sec41}

Here we prove Lemma \ref{lem15}, a general one-dimensional result. Let 
$\pi$ be a probability on $\ZZ_+$, 
denote $Z^\lambda:=\sum_k \lambda^k \pi(k)$ and $\lambda^{\max}:=\sup\{\lambda:Z^\lambda<\infty\}$.
For $\lambda<\lambda^{\max}$ and $I\subset\ZZ_+$, define
\begin{align*}
  \pi^\lambda(k) &:= \frac1{Z^\lambda}\lambda^k \pi(k),\\
  \pi^\lambda(k|I) &:=  \frac{1}{\pi^\lambda(I)}\pi^\lambda(k)\,\one\{k\in I\}.
\end{align*}
We will show that when $I$ is a finite or semi infinite interval,
$$
  \pi^\lambda(\cdott|I)\;\prec\;\pi^{\lambda'}(\cdott|I), \quad 0<\lambda\le \lambda'<\lambda^{\max}. 
$$

Take $\lambda<\lambda^{\max}$ and consider a birth and death process $K^\lambda(t)$ with rates
\begin{align*}
  q^\lambda(k,k+1) &:= \lambda,\\
  q^\lambda(k+1,k)&=q(k+1,k) := \frac{\pi(k)}{\pi(k+1)},\quad k\ge0;
\end{align*}
the rates vanish otherwise. Notice that 
the death rates do not depend on $\lambda$. The measure 
$\pi^\lambda$ satisfies the detailed balance equations,
\begin{align*}
  \pi^\lambda(k)\, q^\lambda(k,k+1)
  &= \frac{\lambda^k}{Z^\lambda}\,\pi(k)\, \lambda \nonumber \\
  &= \frac{\lambda^{k+1}}{Z^\lambda}\, \pi(k+1)\,\frac{\pi(k)}{\pi(k+1)}\\
  &= \pi^\lambda(k+1)\, q^\lambda(k+1,k),
\end{align*}
so that it is reversible for the process $K^\lambda(t)$.
Let  $K^{\lambda,I}(t)$ be the process constrained on $I\subset\ZZ_+$, with rates
  \begin{align*}
  q^{\lambda,I}(k,\tik) := q^{\lambda}(k,\tik)\,\one\{\tik\in I\},\quad k\in I.
  \end{align*}
The conditioned measure $\pi^\lambda(\cdott|I)$ is reversible for $q^{\lambda,I}$, indeed,  for $k,\tik \in I$, 
  \begin{align*}
    \pi^\lambda(k|I)\,q^{\lambda,I}(k,\tik)
    &= \frac{\pi^\lambda(k)}{\pi^\lambda(I)} \,q^{\lambda}(k,\tik) \\
    &= \frac{\pi^\lambda(\tik)}{\pi^\lambda(I)}\,q^{\lambda}(\tik,k)\\
&  = \pi^\lambda(\tik|I)\, q^{\lambda,I}(\tik,k).
\end{align*}
We now couple two birth and death processes with different $\lambda$'s.
\begin{Lemma}
  \label{lem66}
There exists a coupling
\[
  (\tK(t))_{t\ge0}:=(K(t),K'(t))_{t\ge0}
\]
with marginals distributed as $(K^{\lambda,I}(t))_{t\ge0}$ and $(K^{\lambda',I}(t))_{t\ge0}$, respectively, such that 
\begin{align}
    \label{bd3}
    K(0)\le K'(0)\quad\text{implies}\quad  K(t)\le K'(t).
\end{align}
\end{Lemma}

\begin{proof}
  The death rate is the same for both marginals, while the birth rates are $\lambda<\lambda'$.
Coupling: independent marginal jumps when $k\neq k'\in I$. When $k=k'$ both marginals jump down together at rate $q(k,k-1)\,\one\{k-1\in I\}$, they jump up together at rate $\lambda \,\one\{k+1\in I\}$, and the second marginal jumps up alone at rate $(\lambda'-\lambda)\one\{k'+1\in I\}$. The explicit rates, denoted $\tb$, are given for $k,k'\in I$ by
  \begin{alignat}{4}
    \tb\bigl((k,k'),(k+1,k'+1)\bigr) &= \lambda,     &\quad && k &=k',\,      &k+1   &\in I \nonumber \\[-1mm]
    \tb\bigl((k,k'),(k,k'+1)\bigr)&=\lambda'-\lambda,&\quad && k &=k',\,      &k'+1  &\in I \nonumber \\[-1mm]
    \tb\bigl((k,k'),(k-1,k'-1)\bigr) &=q(k,k-1),      &\quad && k &= k' ,\,   & k-1  &\in I \nonumber \\[1mm]
    \tb\bigl((k,k'),(k+1,k')\bigr) &= \lambda,      &\quad && k &\neq k',\,   &k+1   &\in I \nonumber \\[-1mm]
    \tb\bigl((k,k'),(k,k'+1)\bigr)&=\lambda',        &\quad && k &\neq k',\,\,  &k'+1  &\in I \nonumber \\[-1mm]
    \tb\bigl((k,k'),(k,k'-1)\bigr) &= q(k',k'-1),    &\quad && k &\neq k',\,  &k'-1  &\in I \nonumber \\[-1mm]
    \tb\bigl((k,k'),(k-1,k')\bigr) &= q(k,k-1),      &\quad && k &\neq k',\,  &k-1   &\in I,\nonumber
  \end{alignat}
and the rates vanish otherwise. The reader can check that the 
marginal rates are $q^{\lambda,I}$, and $q^{\lambda',I}$, respectively, and that 
jumps violating order have rate zero.
\end{proof}

\begin{proof}[of Lemma \ref{lem15}] Let $h\in I$ and denote $\tE_{h,h}$ be the expectation associated to the coupling $\tK(t)$ with initial condition $\tK(0)=(h,h)$. Denote $\tT$ the return time to $(h,h)$, that is, the hitting time of $(h,h)$ after visiting another state, defined by
  \begin{align*}
    \tT&:= \inf\bigl\{t>0:\tK(t)=(h,h)\text{ and }\tK(t-)\neq(h,h)\bigr \}.
  \end{align*}
The measure $\tpi$ defined by
  \begin{align}
    \label{tmu1}
    \tpi(k,k') &= \frac1{\tE_{h,h} (\tT)}\;\tE_{h,h}\Bigl(\int_0^\tT \,\one\{\tK(t)=(k,k')\}\,dt\Bigr)
\end{align}
is invariant for the coupling $\tK(t)$, see Proposition 2.59 of Liggett  \cite{zbMATH05686748}.  
  By \eqref{bd3}, the integrand in \eqref{tmu1} vanishes when $k>k'$. 
  This implies that $\tpi$ concentrates on $\{(k,k'): k\le k'\}$, implying  \eqref{pi6}.
\end{proof}

\subsection{Order of canonical measures}
\label{sec5}

In this section we show that the canonical measures are stochastically ordered under log-concavity, Theorem \ref{lem16}.  

When the measures $\nu_i$ are absolutely continuous, the result was proven by Efron \cite{MR171335} for a set of functions called PF$_2$ (Polya Frequency functions), which coincides with the set of log-concave functions, as observed by Schoemberg \cite{MR72918}. See also Pitman \cite{MR1429082} and  Theorem 6.1 in Saumard and Wellner \cite{MR3799644}, for the continuous case. Efron mentions that his proof works for the discrete case, which was later proven dynamically, in the framework of queuing theory, see Shantikumar \cite{zbMATH03999013},  Van der Wal \cite{zbMATH04090531} and Daduna and Szekli \cite{zbMATH00885158}.

Liggett \cite{zbMATH01905957} proved the theorem for measures satisfying $\nu_i(x)>0$ for all $x\in \NN_0$, $i\in\{1,\dots,n\}$.
His proof considers zero range processes
$$\sfX(t)=(X_1(t),\dots, X_n(t))\in \NN_0^{\{1,\dots,n\}},$$
where a particle jumps from site $i$ to site $j$ with rate
$$g_i(z)=\frac{\nu_i(z-1)}{\nu_i(z)},$$
where $z$ is the number of particles in site $i$ before the jump.  The canonical measures $\mu_n(\cdott|k)$ are reversible for the process with those rates. Liggett constructs a monotone coupling of two zero range processes $(\sfX(t), \sfX'(t))$ with ordered initial configurations. A monotone coupling conserves the order of the marginals,  a property that holds for the zero range if $g_i$ are increasing for all $i$, a consequence of log-concavity. 

We adapt Liggett's proof for the case $\nu_i(x)>0$ if and only if $x\in \ZZ\cap [0,\max_i]$, where $\max_i\le \infty$, and call it a queuing system with finite waiting room.

\paragraph{Queuing system with finite waiting room}

For each $i$ consider a maximal value $\tmax_i\in\NN\cup\{\infty\}$ and a rate function $g_i:\ZZ\cap [0,\tmax_i]\to \RR_+$, with $g_i(0)=0$. Denote
\begin{align*}
 \cX_n:= (\ZZ\cap [0,\tmax_1])\times \dots \times  (\ZZ\cap [0,\tmax_n]).
\end{align*}
Denote by $\sfx=(x_1,\dots,x_n)$ the elements of $\cX_n$. Define a pure jump continuous time Markov process $\sfX(t)$ on $\cX_n$ with rates 
\begin{align}
  r(\sfx,\sfz)
  &:=
  \begin{cases}
    g(x_i) & \text{if } \sfz = \theta^{ij}\sfx, \text{ and } x_j<\tmax_j,\\
    0&\text{else}.
  \end{cases}\label{rz1},\qquad \sfx,\sfz\in\cX_n,\\
  (\theta^{ij}\sfx)_h&:= x_h-\one\{h=i\}+ \one\{h=j\}, \quad h\in\{1,\dots,n\}. \nonumber
\end{align}
When $\tmax_j=\infty$ for all $j$, the jump rate at $i$ is $g_i(x_i)$, and we have just the non-homogeneous zero range process as in \cite{zbMATH01905957}. Otherwise,  when $x_j= \tmax_j$ the rate of jumping to $j$ is zero, and we have a misanthrope process, see \cite {zbMATH03883377} and \cite{evans-waclaw}. Taking  $\tmax_i\equiv k$ and $g_i(x)\equiv 1$, the misanthrope process coincides with the (homogeneous) $k$-exclusion process, see \cite{zbMATH01359649}.

\paragraph{The Basic coupling}  The basic coupling $\tsfX(t)=(\sfX(t),\sfX'(t))\in\cX_n\times\cX_n$ is defined as the pure jump Markov process with rates
 \begin{align*}
   \tr((\sfx,\sfx'),(\sfz,\sfz'))
    =
      \begin{cases}
        g(x_i)\wedge g(x'_i)
        &\text{if }( \sfz,\sfz')=(\theta^{ij}\sfx,\theta^{ij}\sfx'),  \\
        &x_j < \tmax_j, \, x'_j  <\tmax_j,\\
        g(x_i)\wedge g(x'_i)
        &\text{if }( \sfz,\sfz')=(\sfx,\theta^{ij}\sfx'), \\
        &x_j=\tmax_j, \,x'_j<\tmax_j,\\
        g(x_i)\wedge g(x'_i)
        &\text{if }( \sfz,\sfz')=(\theta^{ij}\sfx,\sfx'), \\
        &x_j<\tmax_j, \,x'_j=\tmax_j,\\
        (g(x'_i)-g(x_i))^+&\text{if } ( \sfz,\sfz')=(\sfx,\theta^{ij}\sfx'),\; x'_j<\tmax_j,\\
        (g(x_i)-g(x'_i))^+&\text{if } ( \sfz,\sfz')=(\theta^{ij}\sfx,\sfx'), \; x_j<\tmax_j,\\
        0&\text{else},
      \end{cases} 
 \end{align*}
The marginals $\sfX(t)$ and $\sfX'(t)$ of the basic coupling are Markov processes with rates \eqref{rz1}. The next lemma shows that the coupling is monotone.

\begin{Lemma}[The basic coupling is monotone]\label{mon}
Assume  $g_i(x)\le g_i(x+1)$ for $x,x+1\in[0,\max_i]$. 
Then, the basic coupling with rates $\tr$ is monotone, that is,
  \begin{align*}
   \text{   $\sfX(0)\le \sfX'(0)$ implies $\sfX(t)\le \sfX'(t)$.}
  \end{align*}
\end{Lemma}
\begin{proof} When $\sfx\le \sfx'$ the rates $\tr$ are zero in the second and fifth lines. In the other three cases the configurations after the jump are ordered $\sfz\le \sfz'$.  
\end{proof}

\paragraph{Canonical measures and queuing system}

Let $\nu_i$ be a family of probability measures on $\ZZ_+$, denote
$\tmax_i:=\sup\{z: \nu_i(z)>0\}\;(\le \infty)$,
assume 
         \begin{align}
          \nu_i(z)&>0,\qquad \text{ for }z\in\ZZ\cap[0,\tmax_i], \label{uiz}
\end{align}
and denote
\begin{align}
  \label{ginu1}
g_{i}(z) :=
  \begin{cases}
    \displaystyle{\frac{\nu_i(z-1)}{\nu_i(z)}}, & z\in\ZZ\cap(0, \tmax_i]\\[1mm]
    0,&\text{else.}
  \end{cases}
\end{align}
The assumption \eqref{uiz} is necessary to construct particle systems with rates \eqref{ginu1} and invariant measure $\mu_n(\cdott|k)$. The condition is satisfied under log-concavity, but the next two lemmas do not require that.  
\begin{Lemma}
  \label{rev4}
Let $\nu_i$ satisfy \eqref{uiz}. Then, for each $k\ge0$ the canonical measure $\mu_n(\cdott|k)$ associated to $\nu_1,\dots,n_n$ is reversible for the process with rates \eqref{rz1} and $g_i$ given by \eqref{ginu1}.
\end{Lemma}
\begin{proof}
Let $\sfx\in\cX_n$ with sum  $x_1+\dots+x_n=k$ and denote $\sfz:= \theta^{ij}\sfx$. It suffices to show that for each $k$, $\mu_n(\cdott|k)$ satisfies the detailed balance equations,
  \begin{align}
    \label{db3}
    \mu_n(\sfx|k) \, g_i(x_i) =  \mu_n(\sfz|k) \, g_j(z_j).
  \end{align}
  Let $x_i=x+1$ and $x_j=y$, so that $z_i= x$, $z_j=y+1$.
Recalling \eqref{can} and  \eqref{ginu1}, \eqref{db3} is equivalent to the identity
\begin{align*}
  &\frac{\bigl(\sprod_{h\notin\{ i,j\}}\nu_h(x)\bigr)\, \nu_i(x+1)\,\nu_j(y)}{\pi_n(k)}\,\frac{\nu_i(x)}{\nu_i(x+1)} \\
   &\quad
    =
  \frac{\bigl(\sprod_{h\notin\{ i,j\}}\nu_h(x)\bigr)\,  \nu_i(x)\,\nu_j(y+1)}{\pi_n(k)}\,\frac{\nu_j(y)}{\nu_j(y+1)}.
\end{align*}
\end{proof}

\begin{proof}[of Theorem \ref{lem16}]
Log-concavity of $\nu_i$ implies that the condition  \eqref{uiz} holds. Then, by Lemma \ref{rev4}, the measure $\mu(\cdott|k)$ is reversible for the process with rates
\begin{align*}
  g_i(y)=\frac{\nu_j(y-1)}{\nu_j(y)}\one\{y\le\tmax_i\}.
\end{align*}
By log-concavity, each $g_i$ is non decreasing in the range $[0,\max_i]$, which in turn implies that there is a monotone coupling with rates~$\tr$, by Lemma \ref{mon}.

Let $k<k'$ and $\tsfX(t):=(\sfX(t),\sfX'(t))$ be a coupling of zero range processes starting with $k$ and $k'$ particles, respectively. Let $A'$ be the event $\{X'_1(0)=k'$ and $X'_j(0)=0$ for all $j\neq1\}$, ``all particles of the $\sfX'$-th marginal are at site $1$'', and similarly $A$ for the $\sfX$ marginal. Denote $\tA:= A\cap A'$, the event ``all particles of both marginals are at site $1$''. Denote the return times to those events by
  \begin{align*}
    \tT&:= \inf\bigl\{t>0:\tsfX(t)\in \tA\text{ and }\tsfX(t-)\notin \tA\bigr \},\\
   T'&:=   \inf\bigl\{t>0:\sfX'(t)\in A'\text{ and }\sfX'(t-)\notin A'\bigr \},
  \end{align*}
Since the coupling is monotone, we have $\tT=T'$ which, using the argument in Lemma \ref{lem15}, shows that the coupling $\tsfX(t)$ has a unique invariant measure $\tmu((\cdot,\cdot)\,|\,k,k')$ with marginals $\mu(\cdott|k)$ and $\mu(\cdott|k')$ concentrating on the set of ordered marginals, $\tmu_n(\sfx\le\sfx'|k,k') =1$, 
that is, $\mu_n(\cdott|k)\; \prec \;\mu_n(\cdott|k')$.
\end{proof}

\section*{Acknowledgments}

We thank support by the Erasmus grant and Neuromat-FAPESP CEPID USP. We thank Milton Jara and Davide Gabrielli for enlightening discussions, and Senya Sloshman and Erwin Bolthausen for mentioning references about the Markovian case. We thank the reviewer for a careful reading and relevant comments.

\bibliographystyle{plain}

\bibliography{largedev}

\end{document}